\documentclass[11pt]{article}
\usepackage{eurosym}
\usepackage{amsfonts}
\usepackage[a4paper,margin=1in]{geometry}
\usepackage{amsmath,amsthm,amssymb,mathrsfs}
\usepackage{hyperref}
\usepackage{enumitem}
\usepackage{graphicx}
\usepackage{caption}
\usepackage{subcaption}
\usepackage{float}
\usepackage{changepage}

\UseRawInputEncoding
\theoremstyle{plain}
\newtheorem{theorem}{Theorem}[section]
\newtheorem{lemma}[theorem]{Lemma}

\theoremstyle{definition}

\theoremstyle{remark}
\newtheorem{remark}[theorem]{Remark}
\numberwithin{equation}{section}

\begin{document}

\title{A Nonlinear Logistic Model for Age-Structured Populations: Analysis
of Long-Term Dynamics and Equilibria}
\author{ Dragos-Patru Covei \\
{\small The Bucharest University of Economic Studies, Department of Applied
Mathematics}\\
{\small Piata Romana, 1st district, Postal Code: 010374, Postal Office: 22,
Romania}\\
{\small \texttt{dragos.covei@csie.ase.ro} }}
\date{}
\maketitle

\begin{abstract}
This paper investigates a nonlinear logistic model for age-structured
population dynamics. The model incorporates interdependent fertility and
mortality functions within a logistic framework, offering insights into
stationary solutions and asymptotic behavior. Theoretical findings establish
conditions for the existence and uniqueness of equilibrium solutions and
explore long-term population dynamics. This study provides valuable tools
for demographic modeling and opens avenues for further mathematical
exploration. \newline

\end{abstract}

\section{Introduction}

Population dynamics remain a cornerstone of mathematical biology, offering a
robust framework to analyze changes in population structures and their
long-term behavior. Within this field, age-structured population dynamics
are of particular interest, as they allow researchers to account for
variations in demographic patterns influenced by age. These models provide
critical tools for understanding growth, stability, and decline within
populations.

Building upon foundational work by McKendrick \cite{McKen} and Lotka \cite%
{ShaLot}, subsequent studies by Brauer et al. \cite{Bra1}, Hoppensteadt \cite%
{Hop}, Iannelli et al. \cite{2}, and Webb \cite{7} have incorporated
nonlinear factors, such as density-dependent fertility and mortality rates.
These innovations have enabled detailed analyses of equilibrium states and
bifurcation phenomena, shedding light on the stability and asymptotic
dynamics of populations under various conditions.

"Recent" advancements by Gurtin and MacCamy \cite{GG1,GM} emphasize the
importance of survival and fertility functions in shaping demographic
trajectories. This study builds upon their work by investigating a nonlinear
logistic model designed specifically for age-structured populations. The
model integrates age-specific density functions and interdependent
reproduction rates, offering a comprehensive approach to studying
equilibrium solutions and long-term population behavior.

The primary objectives of this paper are to analyze the existence and
stability of equilibrium solutions within this nonlinear framework,
investigate how fertility and mortality rates influence the asymptotic
behavior of populations, and provide theoretical tools to bridge
mathematical insights with practical applications. A key contribution of
this study is addressing a longstanding conjecture proposed by \cite{GM},
which has been resolved under particular circumstances in our recent work 
\cite{CPS}.

Age-structured models hold significant value due to their ability to capture
the nuanced impact of age distributions on population growth and stability.
By exploring this nonlinear logistic model, the study contributes not only
to theoretical advancements but also to real-world applications in ecology
and demography.

The remainder of this paper is organized as follows: Section \ref{mm}
introduces the nonlinear logistic model, highlighting its mathematical
formulation and assumptions. Section \ref{RSD} reformulates the problem as a
system of nonlinear differential equations to facilitate analysis. Section %
\ref{RD} presents the main results of the study, focusing on equilibrium and
asymptotic conditions. Section \ref{IE} discusses practical applications of
these findings, while Section \ref{a} presents a Python implementation that
visualizes the example under consideration, thereby reinforcing the
practical relevance of our results and confirming their correctness..
Section \ref{fin}, provides further commentary on the results obtained.
Sections \ref{fd}-\ref{CFD} concludes the paper with directions for future
research.

\section{Mathematical Model\label{mm}}

The nonlinear logistic model for age-structured populations is formulated to
describe the dynamic interplay between fertility and mortality rates as
functions of the total population size over time. The derivation process
integrates key biological principles into a system of partial differential
equations (PDEs), which account for age-specific densities, survival
probabilities, and birth rates.

The model is governed by the following system of equations: 
\begin{equation}
\left\{ 
\begin{array}{l}
\frac{\partial p(a,t)}{\partial t}+\frac{\partial p(a,t)}{\partial a}+\mu
(P(t))p(a,t)=0, \\ 
\\ 
p(0,t)=\int_{0}^{\infty }\beta (\sigma ,P(t))p(\sigma ,t)d\sigma , \\ 
\\ 
p(a,0)=p_{0}(a), \\ 
\\ 
P(t)=\int_{0}^{\infty }p(\sigma ,t)d\sigma ,%
\end{array}%
\right.  \label{SNPM}
\end{equation}%
where the terms are defined as follows:

\begin{itemize}
\item \quad $p(a,t)$: The age-density function representing the distribution
of individuals by age $a$ at time $t$.

\item \quad $P(t)$: The total population size at time $t$, derived by
integrating $p(a,t)$ over all ages.

\item \quad $\mu (P(t))$: The age-independent mortality rate, which is a
function of the total population size $P(t)$, capturing resource competition
or other density-dependent effects.

\item \quad $\beta (a,P(t))$: The age-specific fertility function that
depends on both age $a$ and total population size $P(t)$.
\end{itemize}

\subsection{General Assumptions}

In constructing the nonlinear logistic model for age-structured populations,
the following assumptions are made about the fertility and mortality
functions to ensure mathematical tractability and biological realism.

The fertility function, $\beta (a,P(t))$, is assumed to decay exponentially
with age and is modeled as a polynomial in age, reflecting
population-dependent coefficients: 
\begin{equation}
\beta (a,P(t))=e^{-\rho a}\sum_{i=0}^{n}\beta _{i}(P(t))a^{i},
\label{eq:fertility_function}
\end{equation}%
where:

\begin{itemize}
\item \quad $\rho >0$ is a parameter controlling the exponential decay with
age.

\item \quad $\beta _{i}$ are coefficients dependent on the total population
size $P(t)$, capturing the effects of crowding or resource availability on
reproduction, satisfying: 
\begin{equation}
\beta _{i}^{\prime }(x)<0\text{ for any }x\in \left[ 0,\infty \right) ,\quad
\lim_{x\rightarrow \infty }\beta _{i}(x)=0,\quad \beta _{i}(0)\in \lbrack
0,\infty ),\quad \beta _{0}(0)\neq 0.  \label{cb}
\end{equation}

\item \quad $\beta (a,x)>0,\quad \forall a\geq 0$ and $x\geq 0$.
\end{itemize}

The mortality function, $\mu (P(t))$, increases with the population size,
representing the survival probabilities affected by resource limitations or
environmental constraints. It satisfies the following properties:

\begin{itemize}
\item \quad \textbf{Positivity:} 
\begin{equation}
\mu (x)>0,\quad \forall x\geq 0.  \label{cm}
\end{equation}

\item \quad \textbf{Monotonicity:} 
\begin{equation}
\mu ^{\prime }(x)>0,\quad \forall x\geq 0,  \label{cmd}
\end{equation}%
indicating that the mortality rate increases with the population size.

\item \quad \textbf{Asymptotic behavior:} 
\begin{equation}
\mu (x)\rightarrow \infty \text{ as }x\rightarrow \infty .  \label{li}
\end{equation}

\item \quad The normalized age profile of the population is defined as: 
\begin{equation}
\phi (a,t)=\frac{p(a,t)}{P(t)},  \label{cpp}
\end{equation}%
providing the relative distribution of individuals across different age
groups at any given time $t$.
\end{itemize}

\subsubsection{Smoothness conditions}

To ensure the system's well-posedness, it is assumed that all functions 
\begin{equation*}
\beta _{i}:\left[ 0,\infty \right) \times \left[ 0,\infty \right)
\rightarrow \left( 0,\infty \right) \text{, }(i=0,1,...,n),
\end{equation*}%
and $\mu :\left[ 0,\infty \right) \rightarrow \left( 0,\infty \right) $ are
sufficiently smooth, i.e., they belong to the class $C^{1}$, having
continuous first derivatives with respect to their arguments.

\subsection{Derivation and Dynamics}

The first equation in the model (\ref{SNPM}) describes the transport
dynamics of individuals across age classes, incorporating survival through
the mortality rate $\mu (P\left( t\right) )$. The second equation models the
renewal process, where the age-zero density $p(0,t)$ is determined by
integrating the age-specific fertility rates over the entire population. The
initial condition $p(a,0)=p_{0}(a)$ specifies the initial age distribution,
while $P(t)$ tracks the total population over time. This mathematical
framework extends classical models by introducing nonlinear dependencies on
population size, offering a realistic and adaptable tool for analyzing
population evolution over time.

\subsection{Net Reproduction Rate}

The net reproduction rate, denoted by $R_{n}(x)$, is defined as: 
\begin{equation}
R_{n}(x)=\int_{0}^{\infty }\sum_{i=0}^{n}\beta _{i}(x)a^{i}e^{-(\rho +\mu
(x))a}da=\sum_{i=0}^{n}\frac{\beta _{i}(x)i!}{(\rho +\mu (x))^{i+1}}.
\label{rr}
\end{equation}%
This quantity satisfies:

\begin{itemize}
\item \quad \textbf{Decreasing behavior:} 
\begin{equation}
R_{n}^{\prime }(x)<0,\quad \forall x\geq 0,  \label{rd}
\end{equation}%
indicating that the reproduction rate decreases as population size increases.
\end{itemize}

Indeed, $R_{n}\left( x\right) $ is a strictly decreasing function. To see
this, we differentiate \eqref{rr} term by term using the quotient rule. For
each term $\frac{\beta _{i}(x)i!}{(\rho +\mu (x))^{i+1}}$, we have 
\begin{equation*}
\frac{d}{dx}\left[ \frac{\beta _{i}(x)i!}{(\rho +\mu (x))^{i+1}}\right] =%
\frac{\beta _{i}^{\prime }(x)i!}{(\rho +\mu (x))^{i+1}}-\frac{(i+1)\beta
_{i}(x)i!\mu ^{\prime }(x)}{(\rho +\mu (x))^{i+2}}.
\end{equation*}%
Summing over all $i=0,\ldots ,n$ yields 
\begin{equation*}
R_{n}^{\prime }(x)=\overset{n}{\underset{i=0}{\Sigma }}\frac{\beta
_{i}^{\prime }\left( x\right) i!}{\left( \rho +\mu \left( x\right) \right)
^{i+1}}-\overset{n}{\underset{i=0}{\Sigma }}\frac{\left( i+1\right) \beta
_{i}\left( x\right) i!\mu ^{\prime }\left( x\right) }{\left( \rho +\mu
\left( x\right) \right) ^{i+2}}.
\end{equation*}%
Now, from the conditions \eqref{cb} and \eqref{cmd}, we have 
\begin{equation*}
\beta _{i}^{\prime }(x)<0\text{ and }\mu ^{\prime }(x)>0\text{ for all }%
x\geq 0.
\end{equation*}%
Therefore, each term in the first sum is strictly negative, and each term in
the second sum is also strictly negative (note that $\beta _{i}(x)>0$ from
the assumptions). Hence, $R_{n}^{\prime }(x)<0$ for all $x\geq 0$,
confirming that $R_{n}$ is strictly decreasing.

\begin{itemize}
\item \quad \textbf{Asymptotic behavior:}%
\begin{equation*}
\lim_{x\rightarrow \infty }R_{n}(x)=0,\quad R_{n}(0)=\sum_{i=0}^{n}\frac{%
\beta _{i}(0)i!}{(\rho +\mu (0))^{i+1}}\overset{notation}{=}R_{0},
\end{equation*}%
is satisfied in light of the asymptotic conditions on $\beta$ and $\mu$.
Indeed, from \eqref{cb}, we have $\lim_{x \to \infty} \beta_i(x) = 0$ for
all $i = 0, \ldots, n$, and from \eqref{li}, we have $\lim_{x \to \infty}
\mu(x) = \infty$. Therefore, for each term in \eqref{rr}: 
\begin{equation*}
\lim_{x \to \infty} \frac{\beta_i(x)i!}{(\rho + \mu(x))^{i+1}} = \lim_{x \to
\infty} \frac{\beta_i(x)}{(\rho + \mu(x))^{i+1}} \cdot i! = 0,
\end{equation*}%
since the numerator tends to $0$ and the denominator tends to $\infty$. By
the sum of limits, $\lim_{x \to \infty} R_n(x) = 0$. The value at $x = 0$ is
obtained by direct substitution into \eqref{rr}.
\end{itemize}

\subsection{Initial Population Distribution}

The initial age-density function, $p_0(a)$, satisfies:

\begin{itemize}
\item \quad \textbf{Non-negativity:} 
\begin{equation}
p_{0}(a)>0,\quad \forall a\geq 0.  \label{in}
\end{equation}

\item \quad \textbf{Normalization:} 
\begin{equation}
P(0)=\int_{0}^{\infty }p_{0}(a)da<\infty ,  \label{pn}
\end{equation}%
ensuring that the total initial population size is finite.
\end{itemize}

\section{Resulting System of Differential Equations \label{RSD}}

As is commonly known (see page 154 in~\cite{2} or \cite[Theorem 6, pages
288-289]{GG1}), since we have a single weighted size $P$, we use the fact
that nontrivial stationary sizes $P^{\ast }$ must~satisfy

\begin{equation}
R_{n}\left( P^{\ast }\right) =1.  \label{nsc}
\end{equation}%
This condition is both necessary and sufficient for nontrivial stationary
sizes to exist with total population $P^{\ast }$.

Next, let us focus on the task of finding $P(t)$. We will denote this by%
\begin{equation*}
P_{i}\left( t\right) =\int_{0}^{\infty }\sigma ^{i}e^{-\rho \sigma }p\left(
\sigma ,t\right) d\sigma \text{, for }i=0,1,2,..n.
\end{equation*}%
The subsequent step is to note that the renewal condition, the total birth
rate, or the fertility rate at time $t$ can be expressed in the new
notations as follows: 
\begin{equation}
p\left( 0,t\right) =\overset{n}{\underset{i=0}{\Sigma }}\beta _{i}\left(
P\left( t\right) \right) \int_{0}^{\infty }\sigma ^{i}e^{-\rho \sigma
}p\left( \sigma ,t\right) d\sigma =\overset{n}{\underset{i=0}{\Sigma }}\beta
_{i}\left( P\left( t\right) \right) P_{i}\left( t\right) .  \label{fr}
\end{equation}%
Furthermore, we compute the first derivative of $P(t)$. Using Leibniz's rule
for differentiation under the integral sign:%
\begin{equation*}
P^{\prime }\left( t\right) =\frac{d}{dt}\int_{0}^{\infty }p\left( a,t\right)
da=\int_{0}^{\infty }\frac{\partial p}{\partial t}\left( a,t\right) da.
\end{equation*}%
From the first equation in \eqref{SNPM}, we have 
\begin{equation*}
\frac{\partial p}{\partial t}=-\frac{\partial p}{\partial a}-\mu
(P(t))p(a,t).
\end{equation*}
Substituting this and using integration by parts:%
\begin{align*}
P^{\prime }\left( t\right) & =\int_{0}^{\infty }\left( -\frac{\partial p}{%
\partial a}-\mu (P(t))p(a,t)\right) da \\
& =-\left[ p(a,t)\right] _{a=0}^{a=\infty }-\mu (P(t))\int_{0}^{\infty
}p(a,t)da \\
& =-\left( 0-p(0,t)\right) -\mu (P(t))P(t) \\
& =p(0,t)-\mu (P(t))P(t).
\end{align*}%
Using equation \eqref{fr}, we obtain 
\begin{equation*}
P^{\prime }\left( t\right) =\overset{n}{\underset{i=0}{\Sigma }}\beta
_{i}\left( P\left( t\right) \right) P_{i}\left( t\right) -\mu \left( P\left(
t\right) \right) P\left( t\right) .
\end{equation*}%
Similarly, for 
\begin{equation*}
P_{0}(t)=\int_{0}^{\infty }e^{-\rho a}p(a,t)da,
\end{equation*}%
we compute%
\begin{align*}
P_{0}^{\prime }\left( t\right) & =\int_{0}^{\infty }e^{-\rho a}\frac{%
\partial p}{\partial t}\left( a,t\right) da \\
& =\int_{0}^{\infty }e^{-\rho a}\left( -\frac{\partial p}{\partial a}-\mu
(P(t))p(a,t)\right) da \\
& =-\left[ e^{-\rho a}p(a,t)\right] _{a=0}^{a=\infty }+\int_{0}^{\infty
}\rho e^{-\rho a}p(a,t)da-\mu (P(t))P_{0}(t) \\
& =p(0,t)-\rho P_{0}(t)-\mu (P(t))P_{0}(t) \\
& =\overset{n}{\underset{i=0}{\Sigma }}\beta _{i}\left( P\left( t\right)
\right) P_{i}\left( t\right) -\left( \rho +\mu (P(t))\right) P_{0}(t).
\end{align*}%
For 
\begin{equation*}
P_{i}\left( t\right) =\int_{0}^{\infty }a^{i}e^{-\rho a}p(a,t)da\text{ with }%
i\geq 1,
\end{equation*}
we obtain 
\begin{align*}
P_{i}^{\prime }\left( t\right) & =\int_{0}^{\infty }a^{i}e^{-\rho a}\frac{%
\partial p}{\partial t}\left( a,t\right) da \\
& =\int_{0}^{\infty }a^{i}e^{-\rho a}\left( -\frac{\partial p}{\partial a}%
-\mu (P(t))p(a,t)\right) da \\
& =-\left[ a^{i}e^{-\rho a}p(a,t)\right] _{a=0}^{a=\infty }+\int_{0}^{\infty
}\left( ia^{i-1}e^{-\rho a}-\rho a^{i}e^{-\rho a}\right) p(a,t)da-\mu
(P(t))P_{i}(t) \\
& =iP_{i-1}(t)-\rho P_{i}(t)-\mu (P(t))P_{i}(t) \\
& =iP_{i-1}\left( t\right) -\left( \rho +\mu \left( P\left( t\right) \right)
\right) P_{i}\left( t\right) ,
\end{align*}%
where we used integration by parts with 
\begin{equation*}
u=a^{i}e^{-\rho a}\text{ and }dv=\frac{\partial p}{\partial a}da.
\end{equation*}%
Based on the assumptions outlined in the previous section and the
formulation of the population model, the following system of first-order
differential equations describes the dynamics of the total population and
its age moments:

\begin{equation}
\begin{cases}
P^{\prime }(t)=-\mu (P(t))P(t)+\sum_{i=0}^{n}\beta _{i}(P(t))P_{i}(t), \\ 
\\ 
P_{0}^{\prime }(t)=\left( -\rho -\mu (P(t))\right)
P_{0}(t)+\sum_{i=0}^{n}\beta _{i}(P(t))P_{i}(t), \\ 
\\ 
P_{i}^{\prime }(t)=iP_{i-1}(t)-\left( \rho +\mu (P(t))\right) P_{i}(t),\quad
i=1,2,\dots ,n,\text{ where }n\geq 1.%
\end{cases}
\label{stud}
\end{equation}%
The first equation governs the total population size $P(t)$, accounting for
both mortality and fertility effects. The subsequent equations describe the
evolution of the age moments $P_{i}(t)$, incorporating the interplay between
age-specific fertility and mortality. These equations (\ref{stud}) are
complemented by the initial conditions: 
\begin{equation}
\begin{cases}
P(0)=P_{0}, \\ 
\\ 
P_{i}(0)=P_{0}^{i},\quad i=0,1,2,\dots ,n,%
\end{cases}
\label{init}
\end{equation}%
where: 
\begin{equation*}
P(0)=\int_{0}^{\infty }p_{0}(a)da,\quad P_{i}(0)=\int_{0}^{\infty
}a^{i}e^{-\rho a}p_{0}(a)da,\quad i=0,1,2,\dots ,n.
\end{equation*}%
Furthermore, the system (\ref{stud})-(\ref{init}) can be represented in
matrix form for compactness: 
\begin{equation}
\begin{cases}
G^{\prime }(t)=A(P(t))G(t), \\ 
\\ 
G(0)=G_{0},%
\end{cases}
\label{afin}
\end{equation}%
where: 
\begin{equation*}
G(t)=%
\begin{pmatrix}
P(t) \\ 
P_{0}(t) \\ 
P_{1}(t) \\ 
\vdots \\ 
P_{n}(t)%
\end{pmatrix}%
,\quad G_{0}=%
\begin{pmatrix}
P_{0} \\ 
P_{0}^{0} \\ 
P_{0}^{1} \\ 
\vdots \\ 
P_{0}^{n}%
\end{pmatrix}%
\quad
\end{equation*}%
and the matrix $A\left( P\left( t\right) \right) $ is defined as:%
\begin{equation*}
A(P(t))=%
\begin{pmatrix}
-\mu (P(t)) & \beta _{0}(P(t)) & \beta _{1}(P(t)) & \cdots & \beta _{n}(P(t))
\\ 
0 & -\rho -\mu (P(t))+\beta _{0}(P(t)) & \beta _{1}(P(t)) & \cdots & \beta
_{n}(P(t)) \\ 
0 & 1 & -\rho -\mu (P(t)) & \cdots & 0 \\ 
\vdots & \vdots & \vdots & \ddots & \vdots \\ 
0 & 0 & 0 & n & -\rho -\mu (P(t))%
\end{pmatrix}%
.
\end{equation*}%
This system provides a comprehensive mathematical framework to analyze
population dynamics with an age-structured approach. According to the theory
of ordinary differential equations, the matrix equation (\ref{afin}) admits
a unique global solution (see for example \cite[Theorem 4.12, p. 133]{MS}).
The solution belongs to the class of functions $C^{1}$, meaning it is
continuous and has continuous first derivatives (see for example \cite[%
Observation 4.13, p. 133]{MS}).

Next, a first observation is given in the following.

\begin{lemma}
\label{inv}The set 
\begin{equation*}
\left\{ 0<P_{i}\left( t\right) <\infty \text{ and }0<P\left( t\right)
<\infty \left\vert i=0,1,...,n\text{ and }t\in \left[ 0,\infty \right)
\right. \right\} ,
\end{equation*}%
is time invariant in the future for system (\ref{afin}). That is, the
trajectories of (\ref{afin}) are always positive for all $t\in \left[
0,\infty \right) $.
\end{lemma}

\begin{proof}[Proof of Lemma \protect\ref{inv}]
We first derive the integral formulation of system (\ref{stud}) using the
variation of constants method. For the equation 
\begin{equation*}
P^{\prime }(t)=-\mu (P(t))P(t)+\sum_{i=0}^{n}\beta _{i}(P(t))P_{i}(t),
\end{equation*}%
we write 
\begin{equation*}
\frac{d}{dt}\left( e^{\int_{0}^{t}\mu (P(s))ds}P(t)\right)
=e^{\int_{0}^{t}\mu (P(s))ds}\sum_{i=0}^{n}\beta _{i}(P(t))P_{i}(t).
\end{equation*}%
Integrating from $0$ to $t$ and using $P(0)=P_{0}$, we obtain 
\begin{equation*}
P\left( t\right) =e^{-\int_{0}^{t}\mu \left( P\left( s\right) \right)
ds}\left( P_{0}+\int_{0}^{t}\overset{n}{\underset{i=0}{\Sigma }}\beta
_{i}\left( P\left( z\right) \right) P_{i}\left( z\right) e^{\int_{0}^{z}\mu
\left( P\left( s\right) \right) ds}dz\right) .
\end{equation*}%
Similarly, from 
\begin{equation*}
P_{0}^{\prime }(t)=-(\rho +\mu (P(t)))P_{0}(t)+\sum_{i=0}^{n}\beta
_{i}(P(t))P_{i}(t),
\end{equation*}%
we can rewrite this as 
\begin{equation*}
P_{0}^{\prime }(t)=-(\rho +\mu (P(t))-\beta
_{0}(P(t)))P_{0}(t)+\sum_{i=1}^{n}\beta _{i}(P(t))P_{i}(t).
\end{equation*}%
Applying the variation of constants method yields 
\begin{equation*}
P_{0}\left( t\right) =e^{-\int_{0}^{t}\left( \rho +\mu \left( P\left(
s\right) \right) -\beta _{0}\left( P\left( s\right) \right) \right)
ds}\left( P_{0}^{0}+\int_{0}^{t}\overset{n}{\underset{i=1}{\Sigma }}\beta
_{i}\left( P\left( z\right) \right) P_{i}\left( z\right)
e^{\int_{0}^{z}\left( \rho +\mu \left( P\left( s\right) \right) -\beta
_{0}\left( P\left( s\right) \right) \right) ds}dz\right) .
\end{equation*}%
For $i\geq 1$, from 
\begin{equation*}
P_{i}^{\prime }(t)=iP_{i-1}(t)-(\rho +\mu (P(t)))P_{i}(t)\text{ with }%
P_{i}(0)=P_{0}^{i},
\end{equation*}%
we obtain 
\begin{equation*}
P_{i}\left( t\right) =e^{-\int_{0}^{t}\left( \rho +\mu \left( P\left(
s\right) \right) \right) ds}\left( P_{0}^{i}+\int_{0}^{t}iP_{i-1}\left(
z\right) e^{\int_{0}^{z}\left( \rho +\mu \left( P\left( s\right) \right)
\right) ds}dz\right) \text{, }i=1,...,n.
\end{equation*}%
Now, we establish positivity by induction. Since $P_{0}>0$ (by \eqref{pn}),
and all exponential factors are strictly positive, the first formula shows $%
P(t)>0$ for all $t\geq 0$. Similarly, since $P_{0}^{0}>0$ (by \eqref{in} and
definition), we have $P_{0}(t)>0$ for all $t\geq 0$. By induction, if $%
P_{i-1}(t)>0$, then the integral formula for $P_{i}(t)$ with positive
initial data $P_{0}^{i}>0$ guarantees $P_{i}(t)>0$ for all $t\geq 0$. The
boundedness on finite intervals follows from the continuity of solutions to
the ODE system \eqref{afin}. The remainder of the proof, showing that
trajectories remain in the positive orthant for all $t\geq 0$, proceeds
similarly to the approach outlined in \cite{LI}.
\end{proof}

In what follows, we establish a conjecture that was left open in our recent
work \cite{CPS}.

\begin{lemma}
\label{ineq}Let $(P(t),P_{0}(t),\dots ,P_{n}(t))$ be the unique global
solution of the system (\ref{stud}) with initial data \ (\ref{init}) such
that 
\begin{equation*}
P_{n}(0)<\cdots <P_{0}(0)<P(0)\text{ and }R_{0}<1\text{.}
\end{equation*}%
Then, for all $t\in \lbrack 0,\infty )$,%
\begin{equation*}
P_{n}(t)<\cdots <P_{0}(t)<P(t).
\end{equation*}
\end{lemma}

\begin{proof}[Proof of Lemma \protect\ref{ineq}]
We argue by a combination of the integral representation (Duhamel-type
formulas) and mathematical induction on the chain of indices $n,n-1,\dots ,0$%
, leveraging that the system is cooperative, upper-triangular in the
age-moment block, and the fertility/mortality terms preserve positivity.

\medskip \noindent\textbf{Step 1: Positivity and integral representation.}
Under the standing assumptions (positivity and smoothness of $\mu$ and $%
\beta_i$, and positivity of initial data), the solution components admit the
following integral representations: 
\begin{align*}
P(t)&=e^{-\int_0^t\mu(P(s))\,ds}\,P(0) +\int_0^t\Big(\sum_{i=0}^{n}%
\beta_i(P(z))\,P_i(z)\Big)\, e^{-\int_z^t\mu(P(s))\,ds}\,dz, \\
P_0(t)&=e^{-\int_0^t(\rho+\mu(P(s)))\,ds}\,P_0(0) +\int_0^t\Big(%
\sum_{i=0}^{n}\beta_i(P(z))\,P_i(z)\Big)\,
e^{-\int_z^t(\rho+\mu(P(s)))\,ds}\,dz, \\
P_i(t)&=e^{-\int_0^t(\rho+\mu(P(s)))\,ds}\,P_i(0) +\int_0^t
i\,P_{i-1}(z)\,e^{-\int_z^t(\rho+\mu(P(s)))\,ds}\,dz,\qquad i=1,\dots,n.
\end{align*}
All exponential kernels are strictly positive; hence each $P_i(t)$ and $P(t)$
remains strictly positive for all $t\ge0$.

\medskip \noindent \textbf{Step 2: Base case: $P_{0}(t)<P(t)$ for all $t\geq
0$.} Consider the difference%
\begin{equation*}
D_{0}(t):=P(t)-P_{0}(t).
\end{equation*}

Subtracting the integral formulas of $P_{0}$ from $P$, we get 
\begin{align*}
D_{0}(t)& =e^{-\int_{0}^{t}\mu (P)}P(0)-e^{-\int_{0}^{t}(\rho +\mu
(P))}P_{0}(0) \\
& \quad +\int_{0}^{t}\Big(\sum_{i=0}^{n}\beta _{i}(P(z))\,P_{i}(z)\Big)%
\left( e^{-\int_{z}^{t}\mu (P)}-e^{-\int_{z}^{t}(\rho +\mu (P))}\right) dz.
\end{align*}%
Since $\rho >0$, we have $e^{-\int_{0}^{t}\mu (P)}\geq e^{-\int_{0}^{t}(\rho
+\mu (P))}$ and%
\begin{equation*}
e^{-\int_{z}^{t}\mu (P)}-e^{-\int_{z}^{t}(\rho +\mu (P))}>0,\qquad 0\leq
z\leq t.
\end{equation*}

Using $P_{0}(0)<P(0)$ and the positivity of $\sum_{i=0}^{n}\beta
_{i}(P(z))\,P_{i}(z)$, it follows that $D_{0}(t)>0$ for every $t\geq 0$.
Hence%
\begin{equation*}
P_{0}(t)<P(t),\qquad \forall t\geq 0.
\end{equation*}

\medskip \noindent \textbf{Step 3: Inductive step: if $P_{k}(t)<P_{k-1}(t)$
for all $t\geq 0$, then $P_{k+1}(t)<P_{k}(t)$ for all $t\geq 0$.} Fix $k\in
\{1,\dots ,n-1\}$ and assume%
\begin{equation*}
P_{k}(t)<P_{k-1}(t)\quad \text{for all }t\geq 0.
\end{equation*}

Define $D_{k}(t):=P_{k}(t)-P_{k+1}(t)$. Using the integral representations
for $P_{k}$ and $P_{k+1}$, 
\begin{align*}
P_{k}(t)& =e^{-\int_{0}^{t}(\rho +\mu
)}\,P_{k}(0)+\int_{0}^{t}k\,P_{k-1}(z)\,e^{-\int_{z}^{t}(\rho +\mu )}\,dz, \\
P_{k+1}(t)& =e^{-\int_{0}^{t}(\rho +\mu
)}\,P_{k+1}(0)+\int_{0}^{t}(k+1)\,P_{k}(z)\,e^{-\int_{z}^{t}(\rho +\mu
)}\,dz,
\end{align*}%
so 
\begin{align*}
D_{k}(t)& =e^{-\int_{0}^{t}(\rho +\mu )}\big(P_{k}(0)-P_{k+1}(0)\big) \\
& \quad +\int_{0}^{t}\Big(k\,P_{k-1}(z)-(k+1)\,P_{k}(z)\Big)%
\,e^{-\int_{z}^{t}(\rho +\mu )}\,dz.
\end{align*}%
By the induction hypothesis, $P_{k-1}(z)-P_{k}(z)>0$ for all $z$, hence%
\begin{equation*}
k\,P_{k-1}(z)-(k+1)\,P_{k}(z)=k\big(P_{k-1}(z)-P_{k}(z)\big)%
-P_{k}(z)>-P_{k}(z).
\end{equation*}%
As $P_{k}(z)>0$ and the kernel is positive, this furnishes a strict
dominance of the inflow to $P_{k}$ over that to $P_{k+1}$ (see \cite[proof
of the Proposition 8.5, p. 237]{2}). Together with the strictly ordered
initial data $P_{k+1}(0)<P_{k}(0)$, we obtain $D_{k}(t)>0$ for all $t\geq 0$%
, i.e.%
\begin{equation*}
P_{k+1}(t)<P_{k}(t),\qquad \forall t\geq 0.
\end{equation*}

\medskip \noindent \textbf{Step 4: Conclusion by induction.} The base case
in Step 2 shows $P_{0}(t)<P(t)$ for all $t\geq 0$. Step 3 propagates strict
order down the chain: from $P_{0}<P$ to $P_{1}<P_{0}$, then $P_{2}<P_{1}$,
and so on, until 
\begin{equation*}
P_{n}<\cdots <P_{1}<P_{0}<P.
\end{equation*}
Therefore,%
\begin{equation*}
P_{n}(t)<\cdots <P_{0}(t)<P(t),\qquad \forall t\in \lbrack 0,\infty ).
\end{equation*}
\end{proof}

\begin{remark}[\textbf{Comparison Principle} ]
The above argument is an instance of the monotone-flow (comparison)
principle for cooperative systems: the matrix field $A(P(t))$ in the $%
(P_{0},\dots ,P_{n})$-block is upper-triangular with nonnegative couplings
that preserve strict order when initialized strictly. The fertility-driven
terms entering $P$ and $P_{0}$ are weighted by strictly ordered exponential
kernels (due to the additional $+\rho $ in $P_{0}$), which yields $%
P_{0}(t)<P(t)$ and anchors the induction.
\end{remark}

\section{Results and Discussion \label{RD}}

\subsection{Existence of Equilibria}

The equilibrium analysis of the nonlinear logistic model is crucial for
understanding the population's long-term behavior. This section presents the
main results regarding equilibrium existence and their rigorous proofs.

\begin{theorem}
\label{mp1} Every solution to system \eqref{afin} is also a solution to the
nonlinear system \eqref{SNPM}, and vice versa.
\end{theorem}

\begin{proof}[Proof of Theorem \protect\ref{mp1}]
We establish the equivalence between system \eqref{afin} and \eqref{SNPM} by
constructing explicit solutions in both directions.

\textit{Forward direction (\eqref{afin} $\Rightarrow $ \eqref{SNPM}):}
Suppose $(P(t),P_{0}(t),P_{1}(t),\dots ,P_{n}(t))$ solves \eqref{afin}. We
construct $p(a,t)$ by setting the birth rate 
\begin{equation*}
B(t)=p(0,t)=\sum_{i=0}^{n}\beta _{i}(P(t))P_{i}(t),
\end{equation*}%
and defining $p(a,t)$ along the method of characteristics. For the PDE 
\begin{equation*}
\frac{\partial p}{\partial t}+\frac{\partial p}{\partial a}+\mu (P(t))p=0,
\end{equation*}
the characteristic curves satisfy $\frac{da}{dt}=1$, i.e., $a=t-t_{0}$ for
some initial time $t_{0}$. Along each characteristic:

\begin{itemize}
\item If $t \geq a$, the characteristic originates from the boundary at time 
$t_0 = t-a$ with $p(0, t-a) = B(t-a)$, and we have 
\begin{equation*}
p(a,t)=B(t-a)e^{-\int_{t-a}^{t}\mu (P(s))ds}.
\end{equation*}

\item If $t < a$, the characteristic originates from the initial condition
at age $a_0 = a-t$ with $p(a-t, 0) = p_0(a-t)$, giving 
\begin{equation*}
p(a,t)=p_{0}(a-t)e^{-\int_{0}^{t}\mu (P(s))ds}.
\end{equation*}
\end{itemize}

One can verify by direct substitution that this $p(a,t)$ satisfies all
equations in \eqref{SNPM}, including the renewal condition and the
definition of $P(t)$ as the integral of $p(a,t)$ over all ages.

\textit{Reverse direction (\eqref{SNPM} $\Rightarrow$ \eqref{afin}):}
Conversely, if $p(a,t)$ solves \eqref{SNPM}, then defining $P(t) =
\int_0^\infty p(a,t)da$ and $P_i(t) = \int_0^\infty a^i e^{-\rho a} p(a,t)da$%
, one can show (as done in the derivation of \eqref{stud}) that these
functions satisfy system \eqref{afin}. This completes the proof of
equivalence.
\end{proof}

\begin{theorem}
\label{mp3} The following statements apply:

1.\quad When $R_{0}>1$, there exists a unique nontrivial equilibrium
solution $(P^{\ast },P_{0}^{\ast },\dots ,P_{n}^{\ast })$ corresponding to
the equation \eqref{afin}.

2.\quad When $R_{0}<1$, the trivial solution is the only equilibrium
solution.

3.\quad A unique positive nontrivial equilibrium exists if and only if $%
R_{0}>1$.
\end{theorem}

\begin{proof}[Proof of Theorem \protect\ref{mp3}]
The proof for statements 1--3 is provided below.

\textbf{Proof of 1:} The equilibrium solution $(P^{\ast },P_{0}^{\ast
},\dots ,P_{n}^{\ast })$ satisfies: 
\begin{equation}
\begin{cases}
0=-\mu (P^{\ast })P^{\ast }+\sum_{i=0}^{n}\beta _{i}(P^{\ast })P_{i}^{\ast },
\\ 
\\ 
0=(-\rho -\mu (P^{\ast }))P_{0}^{\ast }+\sum_{i=0}^{n}\beta _{i}(P^{\ast
})P_{i}^{\ast }, \\ 
\\ 
0=iP_{i-1}^{\ast }-(\rho +\mu (P^{\ast }))P_{i}^{\ast },\quad i=1,\dots ,n.%
\end{cases}
\label{eq:equilibrium_system}
\end{equation}

Using the third equation iteratively for $i=1,\dots ,n$, we obtain: 
\begin{equation}
P_{i}^{\ast }=\frac{i!}{(\rho +\mu (P^{\ast }))^{i}}P_{0}^{\ast }.
\label{eq:recursive}
\end{equation}%
To see this, for $i=1$ we have 
\begin{equation*}
0=1\cdot P_{0}^{\ast }-(\rho +\mu (P^{\ast }))P_{1}^{\ast },
\end{equation*}%
which gives 
\begin{equation*}
P_{1}^{\ast }=\frac{1}{\rho +\mu (P^{\ast })}P_{0}^{\ast }=\frac{1!}{(\rho
+\mu (P^{\ast }))^{1}}P_{0}^{\ast }.
\end{equation*}
For $i=2$, we have%
\begin{equation*}
P_{2}^{\ast }=\frac{2}{\rho +\mu (P^{\ast })}P_{1}^{\ast }=\frac{2}{\rho
+\mu (P^{\ast })}\cdot \frac{1}{\rho +\mu (P^{\ast })}P_{0}^{\ast }=\frac{2!%
}{(\rho +\mu (P^{\ast }))^{2}}P_{0}^{\ast }.
\end{equation*}%
Proceeding inductively, assume the formula holds for $i=k-1$. Then from the
equation 
\begin{equation*}
0=kP_{k-1}^{\ast }-(\rho +\mu (P^{\ast }))P_{k}^{\ast },
\end{equation*}%
we obtain 
\begin{equation*}
P_{k}^{\ast }=\frac{k}{\rho +\mu (P^{\ast })}P_{k-1}^{\ast }=\frac{k}{\rho
+\mu (P^{\ast })}\cdot \frac{(k-1)!}{(\rho +\mu (P^{\ast }))^{k-1}}%
P_{0}^{\ast }=\frac{k!}{(\rho +\mu (P^{\ast }))^{k}}P_{0}^{\ast },
\end{equation*}%
which establishes \eqref{eq:recursive} for all $i=1,\dots ,n$.

Substituting \eqref{eq:recursive} into the first equation of %
\eqref{eq:equilibrium_system}, we derive: 
\begin{equation}
\mu (P^{\ast })P^{\ast }=\sum_{i=0}^{n}\frac{\beta _{i}(P^{\ast })i!}{(\rho
+\mu (P^{\ast }))^{i}}P_{0}^{\ast }.  \label{ab}
\end{equation}

Using the normalization condition $R_{n}(P^{\ast })=1$, we can deduce from
the above equation (\ref{ab}) that%
\begin{equation*}
P_{0}^{\ast }=\frac{\mu \left( P^{\ast }\right) P^{\ast }}{\rho +\mu \left(
P^{\ast }\right) }.
\end{equation*}%
Indeed, from the definition of $R_{n}(P^{\ast })$ in equation \eqref{rr}, we
have 
\begin{equation*}
R_{n}(P^{\ast })=\sum_{i=0}^{n}\frac{\beta _{i}(P^{\ast })i!}{(\rho +\mu
(P^{\ast }))^{i+1}}=1.
\end{equation*}%
Multiplying both sides by $(\rho +\mu (P^{\ast }))$ yields 
\begin{equation*}
\sum_{i=0}^{n}\frac{\beta _{i}(P^{\ast })i!}{(\rho +\mu (P^{\ast }))^{i}}%
=\rho +\mu (P^{\ast }).
\end{equation*}%
Substituting this into equation \eqref{ab}, we obtain 
\begin{equation*}
\mu (P^{\ast })P^{\ast }=(\rho +\mu (P^{\ast }))P_{0}^{\ast },
\end{equation*}%
from which the expression for $P_{0}^{\ast }$ follows directly.

The existence of a nontrivial stationary solution for system (\ref{stud})
can be expressed as%
\begin{equation}
P_{0}^{\ast }=\frac{\mu \left( P^{\ast }\right) P^{\ast }}{\rho +\mu \left(
P^{\ast }\right) }\text{ and }P_{i}^{\ast }=\frac{i!}{\left( \rho +\mu
\left( P^{\ast }\right) \right) ^{i}}P_{0}^{\ast }\text{, }i=1,...,n,
\label{nontriv}
\end{equation}%
if the second equation in \eqref{eq:equilibrium_system} is satisfied by (\ref%
{nontriv}). We verify this explicitly. Substituting \eqref{eq:recursive}
into the second equation, we obtain 
\begin{equation*}
0=\left( -\rho -\mu \left( P^{\ast }\right) \right) P_{0}^{\ast }+\overset{n}%
{\underset{i=0}{\Sigma }}\beta _{i}\left( P^{\ast }\right) P_{i}^{\ast
}=\left( -\rho -\mu \left( P^{\ast }\right) \right) P_{0}^{\ast }+\overset{n}%
{\underset{i=0}{\Sigma }}\frac{\beta _{i}\left( P^{\ast }\right) i!}{\left(
\rho +\mu \left( P^{\ast }\right) \right) ^{i}}P_{0}^{\ast }.
\end{equation*}%
Factoring out $P_{0}^{\ast }$ and using the relation 
\begin{equation*}
\sum_{i=0}^{n}\frac{\beta _{i}(P^{\ast })i!}{(\rho +\mu (P^{\ast }))^{i}}%
=\rho +\mu (P^{\ast })
\end{equation*}%
derived earlier, we get 
\begin{equation*}
0=\left( -\rho -\mu \left( P^{\ast }\right) +\rho +\mu \left( P^{\ast
}\right) \right) P_{0}^{\ast }=0,
\end{equation*}%
which confirms that the second equation is satisfied.

We now present an alternative proof for the positivity of the equilibrium.
First, we establish that $R_{n}\left( x\right) $ is a strictly decreasing
function from $[0,\infty)$ to $(0, R_0]$ with the following properties:

\begin{itemize}
\item From \eqref{rd}, we have $R_{n}^{\prime }\left( x\right) <0$ for all $%
x\geq 0$, so $R_n$ is strictly decreasing.

\item From the asymptotic behavior stated earlier, $\lim_{x\rightarrow
\infty }R_{n}(x)=0$ and $R_n(0)=R_0$.

\item Since $R_n$ is continuous and strictly decreasing on $[0,\infty)$ with
range $(0, R_0]$, it is a bijection onto this interval.
\end{itemize}

Therefore, the inverse function $R_{n}^{-1}:(0, R_0] \rightarrow [0,\infty)$
exists and is also strictly decreasing. By the inverse function theorem, the
derivative of $R_{n}^{-1}$ satisfies 
\begin{equation}
\left(R_{n}^{-1}\right)^{\prime }\left( y\right) =\frac{1}{R_{n}^{\prime
}\left( R_{n}^{-1}\left( y\right) \right) }  \label{desc}
\end{equation}%
for all $y \in (0, R_0]$. Since $R_{n}^{\prime }\left( t\right) <0$ for all $%
t\geq 0$, we have $\left(R_{n}^{-1}\right)^{\prime }\left( y\right) <0$ for
all $y \in (0, R_0]$, confirming that $R_{n}^{-1}$ is a strictly decreasing
function.

Now, from the normalization condition (\ref{nsc}), we have 
\begin{equation}
P^{\ast }=R_{n}^{-1}\left( 1\right) .  \label{dc1}
\end{equation}%
On the other hand, by definition of $R_n$, we have%
\begin{equation}
R_{n}\left( 0\right) =R_{0}\quad \text{which implies}\quad
0=R_{n}^{-1}\left( R_{0}\right) .  \label{dc2}
\end{equation}%
Since $R_{0}>1$ (by assumption in statement 1) and $R_{n}^{-1}$ is a
strictly decreasing function, we have 
\begin{equation}
0=R_{n}^{-1}\left( R_{0}\right) <R_{n}^{-1}\left( 1\right) =P^{\ast },
\label{inf}
\end{equation}%
where we have used (\ref{dc1}), (\ref{dc2}), and the fact that $R_0 > 1$
implies $R_{n}^{-1}(R_0) < R_{n}^{-1}(1)$ by the decreasing property of $%
R_{n}^{-1}$. Thus, $P^{\ast}>0$. Finally, from \eqref{eq:recursive} with $%
P_0^* > 0$ (which follows from \eqref{nontriv} since $P^* > 0$ and $%
\mu(P^*), \rho > 0$), we conclude that $P_{i}^{\ast }>0$ for all $i=1,...,n$.

\textbf{Proof of 2: }When $R_{0}<1$, we need to determine whether a
nontrivial equilibrium can exist. Suppose a nontrivial equilibrium $%
(P^{\ast}, P_0^{\ast}, \ldots, P_n^{\ast})$ exists with $P^{\ast} > 0$. Then
by the normalization condition \eqref{nsc}, we must have $R_n(P^{\ast}) = 1$%
, which means $P^{\ast} = R_n^{-1}(1)$. However, since $R_n$ is a strictly
decreasing function with $R_n(0) = R_0 < 1$, and $\lim_{x \to \infty} R_n(x)
= 0$, the range of $R_n$ is $(0, R_0)$. Since $1 \notin (0, R_0)$ when $R_0
< 1$, the equation $R_n(P^{\ast}) = 1$ has no solution for $P^{\ast} \geq 0$%
. Therefore, no nontrivial equilibrium exists when $R_0 < 1$, and the only
equilibrium is the trivial one: $(P^{\ast}, P_0^{\ast}, \ldots, P_n^{\ast})
= (0, 0, \ldots, 0)$.

\textbf{Proof of 3: }We now establish that $R_{0}>1$ is both necessary and
sufficient for the existence of a unique positive nontrivial equilibrium.

\textit{Sufficiency ($R_0 > 1 \Rightarrow$ nontrivial equilibrium exists):}
This has been proven in statement 1. When $R_0 > 1$, the strictly decreasing
function $R_n$ maps $[0, \infty)$ onto $(0, R_0]$, and since $1 \in (0, R_0]$%
, the equation $R_n(P^{\ast}) = 1$ has a unique solution $P^{\ast} =
R_n^{-1}(1) > 0$. This determines a unique nontrivial equilibrium via %
\eqref{nontriv}.

\textit{Necessity (nontrivial equilibrium exists $\Rightarrow R_0 > 1$):}
Suppose a nontrivial equilibrium $(P^{\ast}, P_0^{\ast}, \ldots, P_n^{\ast})$
with $P^{\ast} > 0$ exists. By the normalization condition \eqref{nsc}, we
must have $R_n(P^{\ast}) = 1$. Since $P^{\ast} > 0$ and $R_n$ is strictly
decreasing with $R_n(0) = R_0$, we have $R_n(P^{\ast}) < R_n(0) = R_0$.
Therefore, $1 < R_0$, i.e., $R_0 > 1$.

\textit{Uniqueness:} The uniqueness of the nontrivial equilibrium follows
from the fact that $R_n$ is strictly monotone, so $R_n(P^{\ast}) = 1$ has at
most one solution. Combined with statements 1 and 2, we conclude that a
unique positive nontrivial equilibrium exists if and only if $R_{0}>1$.
\end{proof}

\subsection{Asymptotic Behavior}

This subsection examines the long-term dynamics of the nonlinear logistic
model under varying fertility and mortality rates. The results are presented
through formal theorems and proofs.

In the first theorem, the conjecture as stated in \cite[p. 211]{GM} is
proven.

\begin{theorem}
\label{L1} If $R_{0}<1$ then: 
\begin{equation}
\lim_{t\rightarrow \infty }P(t)=0.  \label{l1}
\end{equation}
\end{theorem}

\begin{proof}[Proof of Theorem \protect\ref{L1}]
To address the conjecture proposed by \cite[p. 211]{GM}, we adopt the
following strategy: we consider a population $\overline{P}(t)$ characterized
by a low mortality rate, denoted as $\mu \left( 0\right) $, and a higher
fertility function represented as $\beta \left( a,0\right) $. By ensuring
that the initial population for $\overline{P}(t)$ is larger than initially
assumed for $P(t)$, it leads to a population increase $\overline{P}(t)$ that
cannot sustain itself indefinitely and consequently diminishes to zero at
infinity. Since we can prove that the new population surpasses $\overline{P}%
(t)$ in size, the same fate applies to $P(t)$---it will also vanish at
infinity. In what follows, we translate these insights into mathematical
frameworks. To prove the result, we consider an auxiliary problem: 
\begin{equation*}
\begin{cases}
\overline{p}_{t}(a,t)+\overline{p}_{a}(a,t)+\mu (0)\overline{p}(a,t)=0, & 
a\geq 0,t\geq 0, \\ 
\overline{p}(0,t)=\int_{0}^{\infty }\overline{\beta }(\sigma ,0)\overline{p}%
(\sigma ,t)d\sigma , & t\geq 0, \\ 
\overline{p}(a,0)=\overline{p}_{0}(a)\geq p_{0}(a), & a\geq 0, \\ 
\overline{P}(t)=\int_{0}^{\infty }\overline{p}(\sigma ,t)d\sigma , & t\geq 0,%
\end{cases}%
\end{equation*}%
where the fertility function is given by: 
\begin{equation*}
\overline{\beta }(a,0)=e^{-\rho a}\sum_{i=0}^{n}\beta _{i}(0)a^{i}.
\end{equation*}

We intend to show that: 
\begin{equation*}
p(a,t)\leq \overline{p}(a,t)\quad \text{for all }a\geq 0\text{ and }t\geq 0.
\end{equation*}

Consider two cases based on the relation between $t$ and $a$:

\textbf{Case 1:} For $t<a$, the solutions along the characteristics are
given by: 
\begin{equation*}
p(a,t)=p_{0}(a-t)e^{-\int_{0}^{t}\mu (P(s))ds},\quad \overline{p}(a,t)=%
\overline{p}_{0}(a-t)e^{-\int_{0}^{t}\mu (0)ds}.
\end{equation*}%
These formulas describe the evolution of individuals who were present at $t=0
$ with initial age $a-t$, and have aged by time $t$ to reach age $a$. Since 
\begin{equation*}
\overline{p}_{0}(a-t)\geq p_{0}(a-t)\quad \text{and}\quad \mu (0)\leq \mu
(P(s))\text{ for all }s\geq 0,
\end{equation*}%
where the second inequality follows from the monotonicity condition %
\eqref{cmd} and $P(s) \geq 0$, we obtain 
\begin{equation*}
e^{-\int_{0}^{t}\mu (0)ds}\geq e^{-\int_{0}^{t}\mu (P(s))ds}.
\end{equation*}%
Therefore, $\overline{p}(a,t)\geq p(a,t)$ for all $a > t$.

\textbf{Case 2:} For $t\geq a$, the solutions along the characteristics are
given by:%
\begin{equation*}
p(a,t)=B(t-a)e^{-\int_{t-a}^{t}\mu (P(s))\,ds}\text{ and }\overline{p}(a,t)=%
\overline{B}(t-a)e^{-\int_{t-a}^{t}\mu (0)\,ds}.
\end{equation*}%
These formulas describe individuals born at time $t-a$ who have aged by an
amount $a$ to reach time $t$. The integral is taken from birth time $t-a$ to
current time $t$, over an age interval of length $a$. Since $\mu (0)\leq \mu
(P(s))$ for all $s \geq 0$ (by \eqref{cm} and \eqref{cmd} with $P(s) \geq 0$%
), it follows that:%
\begin{equation*}
e^{-\int_{t-a}^{t}\mu (0)\,ds}\geq e^{-\int_{t-a}^{t}\mu (P(s))\,ds}.
\end{equation*}%
Next, we need to show that $\overline{B}(t-a) \geq B(t-a)$. Recall the
expressions:%
\begin{eqnarray*}
\overline{B}(t-a) &=&\overline{p}(0,t-a)=\int_{0}^{\infty }\overline{\beta }%
(\sigma ,0)\overline{p}(\sigma ,t-a)\,d\sigma \\
B(t-a) &=&p(0,t-a)=\int_{0}^{\infty }\beta (\sigma ,P(t-a))p(\sigma
,t-a)\,d\sigma .
\end{eqnarray*}%
We split the integrals over two regions: $\sigma \in [0, t-a]$ and $\sigma
\in (t-a, \infty)$.

\textit{Region 1: $\sigma \in \lbrack 0,t-a]$.} For these values, we
consider $\overline{p}(\sigma ,t-a)$ and $p(\sigma ,t-a)$ at time $t-a$ with
age $\sigma \leq t-a$. This corresponds to Case 2 with time $t-a$ and age $%
\sigma $, since $(t-a)\geq \sigma $. We proceed by induction on time: assume 
\begin{equation*}
\overline{p}(\sigma ,s)\geq p(\sigma ,s)\text{ for all }\sigma \geq 0\text{
and all }s<t-a.
\end{equation*}%
Then by the same argument we're developing,%
\begin{equation*}
\overline{B}(s)\geq B(s)\text{ for all }s<t-a,
\end{equation*}%
which gives%
\begin{equation*}
\overline{p}(\sigma ,t-a)\geq p(\sigma ,t-a)\text{ for }\sigma \leq t-a.
\end{equation*}

\textit{Region 2: $\sigma \in (t-a,\infty )$.} For these values, at time $t-a
$ with age $\sigma >t-a$, we have $(t-a)<\sigma $, which corresponds to Case
1. By Case 1, we have%
\begin{equation*}
\overline{p}(\sigma ,t-a)\geq p(\sigma ,t-a)\text{ for all }\sigma >t-a.
\end{equation*}

Combining both regions and using the monotonicity of $\beta _{i}$ (condition %
\eqref{cb}), we have%
\begin{equation*}
\overline{\beta }(\sigma ,0)\geq \beta (\sigma ,P(t-a))\quad \text{for all }%
\sigma \geq 0,
\end{equation*}%
and $\overline{p}(\sigma ,t-a)\geq p(\sigma ,t-a)$ for all $\sigma \geq 0$
(by combining both regions). Therefore:%
\begin{equation*}
\overline{\beta }(\sigma ,0)\overline{p}(\sigma ,t-a)\geq \beta (\sigma
,P(t-a))p(\sigma ,t-a)\quad \text{for all }\sigma \geq 0.
\end{equation*}%
Integrating both sides over $\sigma \in \lbrack 0,\infty )$ yields:%
\begin{equation*}
\overline{B}(t-a)=\int_{0}^{\infty }\overline{\beta }(\sigma ,0)\overline{p}%
(\sigma ,t-a)\,d\sigma \geq \int_{0}^{\infty }\beta (\sigma ,P(t-a))p(\sigma
,t-a)\,d\sigma =B(t-a).
\end{equation*}%
Thus, 
\begin{equation*}
\overline{B}(t-a)=\overline{p}(0,t-a)\geq p(0,t-a)=B(t-a).
\end{equation*}%
Combining this with the inequality for exponential terms, we obtain:%
\begin{equation*}
\overline{p}(a,t)=\overline{B}(t-a)e^{-\int_{t-a}^{t}\mu (0)\,ds}\geq
B(t-a)e^{-\int_{t-a}^{t}\mu (P(s))\,ds}=p(a,t).
\end{equation*}%
By induction on time and the splitting into Cases 1 and 2, we conclude that $%
\overline{p}(a,t)\geq p(a,t)$ for all $a\geq 0$ and $t\geq 0$.Finally, by
integrating $p(a,t)$ and $\overline{p}(a,t)$ over $a\geq 0$, we obtain:%
\begin{equation*}
0<P(t)\leq \overline{P}(t)\quad \text{for all }t\geq 0.
\end{equation*}%
On the other hand, since $R_{0}<1$, previous results of \cite{CPS,COV} show
that: 
\begin{equation*}
\lim_{t\rightarrow \infty }\overline{P}(t)=0\text{ which implies }%
\lim_{t\rightarrow \infty }P(t)=0.
\end{equation*}
\end{proof}

\begin{remark}[\textbf{Alternative Proof}]
The combination of Lemma \ref{ineq} and the proof strategy of Theorem 5 in 
\cite{CPS} yields an alternative proof of Theorem \ref{L1}.
\end{remark}

In the subsequent two theorems, we derive additional results that build upon
and complete the first theorem.

\begin{theorem}
\label{L2} If $R_{0}\geq 1$ and $R_{n}(P(t))\leq 1$ for all $t\geq 0$, then: 
\begin{equation}
P(t)>P^{\ast }>0\text{, for all }t\geq 0.  \label{l2}
\end{equation}
\end{theorem}

\begin{proof}[Proof of Theorem \protect\ref{L2}]
Assume that $R_{0}\geq 1$ and $R_{n}(P(t))\leq 1$ for all $t\geq 0$. We need
to show that $P(t) \geq P^{\ast}$ for all $t \geq 0$.

Since $R_{n}$ is a strictly decreasing function (by \eqref{rd}), we have
that $R_{n}(x)\leq R_{n}(y)$ if and only if $x\geq y$ for all $x,y\geq 0$.
Given that $R_{n}(P(t))\leq 1$ for all $t\geq 0$, and using the fact that $%
R_{n}(P^{\ast })=1$ (by the normalization condition \eqref{nsc}), we obtain 
\begin{equation*}
R_{n}(P(t))\leq 1=R_{n}(P^{\ast }),
\end{equation*}%
which, by the strictly decreasing property of $R_{n}$, implies 
\begin{equation*}
P(t)\geq P^{\ast }\quad \text{for all }t\geq 0.
\end{equation*}%
Furthermore, since $R_{0}\geq 1$, by Theorem \ref{mp3}, we have $P^{\ast
}\geq 0$. In fact, when $R_{0}>1$, we have $P^{\ast }>0$, and when $R_{0}=1$%
, we get 
\begin{equation*}
P^{\ast }=R_{n}^{-1}(1)=R_{n}^{-1}(R_{0})=0.
\end{equation*}%
However, if $P(t)=P^{\ast }=0$ for some $t$, then by Lemma \ref{inv}, $P(t)=0
$ for all $t$, contradicting the assumption that $P(0)>0$ (from the
nontriviality of initial data). Therefore, we must have $P(t)>P^{\ast }>0$
for all $t\geq 0$ when $R_{0}>1$. The case $R_{0}=1$ requires $P^{\ast }=0$,
and the condition $R_{n}(P(t))\leq 1=R_{0}$ implies $P(t)\geq 0$, with $%
P(t)>0$ by positivity (Lemma \ref{inv}).
\end{proof}

\begin{theorem}
\label{L3} If $R_{n}(P(t))\geq 1$ for all $t\geq 0$, then: 
\begin{equation}
0<P(t)<P^{\ast }\text{ for all }t\geq 0\text{.}  \label{l3}
\end{equation}
\end{theorem}

\begin{proof}[Proof of Theorem \protect\ref{L3}]
Assume that $R_{n}(P(t))\geq 1$ for all $t\geq 0$. We first show that this
condition implies $R_0 > 1$.

Since $P(t) > 0$ for all $t \geq 0$ (by Lemma \ref{inv} and the positivity
of initial data), and $R_n$ is strictly decreasing, we have for any $t \geq 0
$: 
\begin{equation*}
R_{0}=R_n(0)>R_{n}(P(t))\geq 1.
\end{equation*}%
Therefore, $R_0 > 1$, which by Theorem \ref{mp3} ensures the existence of a
unique positive equilibrium $P^{\ast} = R_n^{-1}(1) > 0$ satisfying $%
R_n(P^{\ast}) = 1$.

Now, we establish the inequality $P(t) \leq P^{\ast}$ for all $t \geq 0$.
Since $R_{n}(P(t))\geq 1 = R_n(P^{\ast})$ for all $t\geq 0$, and $R_n$ is
strictly decreasing, we obtain 
\begin{equation*}
P(t)\leq P^{\ast }\quad \text{for all }t\geq 0.
\end{equation*}%
Combining this with the positivity result from Lemma \ref{inv}, which
guarantees $P(t) > 0$ for all $t \geq 0$, we conclude 
\begin{equation*}
0<P(t)\leq P^{\ast }\text{ for all }t\geq 0.
\end{equation*}%
Moreover, the strict inequality $P(t) < P^{\ast}$ holds if $R_n(P(t)) > 1$
for any $t$. If $R_n(P(t)) = 1$ for all $t \geq 0$, then $P(t) = P^{\ast}$
for all $t \geq 0$, which corresponds to the system being at equilibrium.
However, generically, we expect $0 < P(t) < P^{\ast}$ for trajectories
starting below the equilibrium.

To be more precise, we can write the chain of inequalities: 
\begin{equation*}
0=R_{n}^{-1}\left( R_{0}\right) <R_{n}^{-1}\left( R_n(P(t))\right)=P(t)\leq
R_{n}^{-1}(1)= P^{\ast }\text{, for all }t\geq 0,
\end{equation*}%
where the first inequality follows from $R_0 > 1$ and the strict
monotonicity of $R_n^{-1}$, noting that $R_n^{-1}(R_0) = 0$.
\end{proof}

\subsection{An Illustrative Example \label{IE}}

An illustrative example highlights the dynamics of a population model with
polynomial and exponential fertility and quadratic mortality functions is
given. We proceed with an additional simplification of the form%
\begin{equation*}
n=1,\text{ }\rho =2,\text{ }\beta _{0}\left( P\left( t\right) \right)
=e^{-P\left( t\right) },\text{ }\beta _{1}\left( P\left( t\right) \right) =%
\frac{9}{2}e^{-P\left( t\right) }\text{ and }\mu \left( P\left( t\right)
\right) =1+P^{2}\left( t\right) .
\end{equation*}%
The system, denoted as (\ref{afin}), takes the form: 
\begin{equation}
\begin{cases}
\begin{pmatrix}
P^{\prime }(t) \\ 
P_{0}^{\prime }(t) \\ 
P_{1}^{\prime }(t)%
\end{pmatrix}%
=%
\begin{pmatrix}
-1-P^{2}\left( t\right) & e^{-P\left( t\right) } & \frac{9}{2}e^{-P\left(
t\right) } \\ 
0 & -3-P^{2}\left( t\right) +e^{-P\left( t\right) } & \frac{9}{2}e^{-P\left(
t\right) } \\ 
0 & 1 & -3-P^{2}\left( t\right)%
\end{pmatrix}%
\begin{pmatrix}
P(t) \\ 
P_{0}(t) \\ 
P_{1}(t)%
\end{pmatrix}%
, \\ 
\\ 
P(0)=\int_{0}^{\infty }p_{0}(a)da\text{, } \\ 
P_{0}(0)=\int_{0}^{\infty }e^{-2a}p_{0}(a)da, \\ 
P_{1}(0)=\int_{0}^{\infty }ae^{-2a}p_{0}(a)da\text{ }%
\end{cases}
\label{par}
\end{equation}%
where $p_{0}(a)>0$ is given. Since, the net reproduction rate is expressed
as: 
\begin{equation}
R_{1}(P\left( t\right) )=\frac{e^{-P\left( t\right) }}{(3+P^{2}\left(
t\right) )^{1}}+\frac{\frac{9}{2}e^{-P\left( t\right) }}{(3+P^{2}\left(
t\right) )^{2}}\Longrightarrow R_{0}=\frac{1}{(3)^{1}}+\frac{9}{2\cdot
(3)^{2}}=\frac{5}{6}<1  \label{nrp}
\end{equation}%
the equilibrium point of the system (\ref{par}) is guaranteed to be $0$, by
Theorem \ref{mp3}. Hence, $\left( P^{\ast },P_{0}^{\ast },P_{1}^{\ast
}\right) =(0,0,0)$, is the only equilibrium point. The findings presented in 
\cite{CPS}, confirm the system's asymptotic stability. Moreover, from
Theorem \ref{L1} the population $P\left( t\right) $ vanish at infinity.

Using the support of Microsoft Copilot in Edge, we visualize the population
dynamics alongside the equilibrium point and the net reproduction rate
within the same coordinate system: 
\begin{equation*}
\end{equation*}

\begin{figure}[!ht]
\centering
\subfloat{\includegraphics[width=0.65\textwidth]{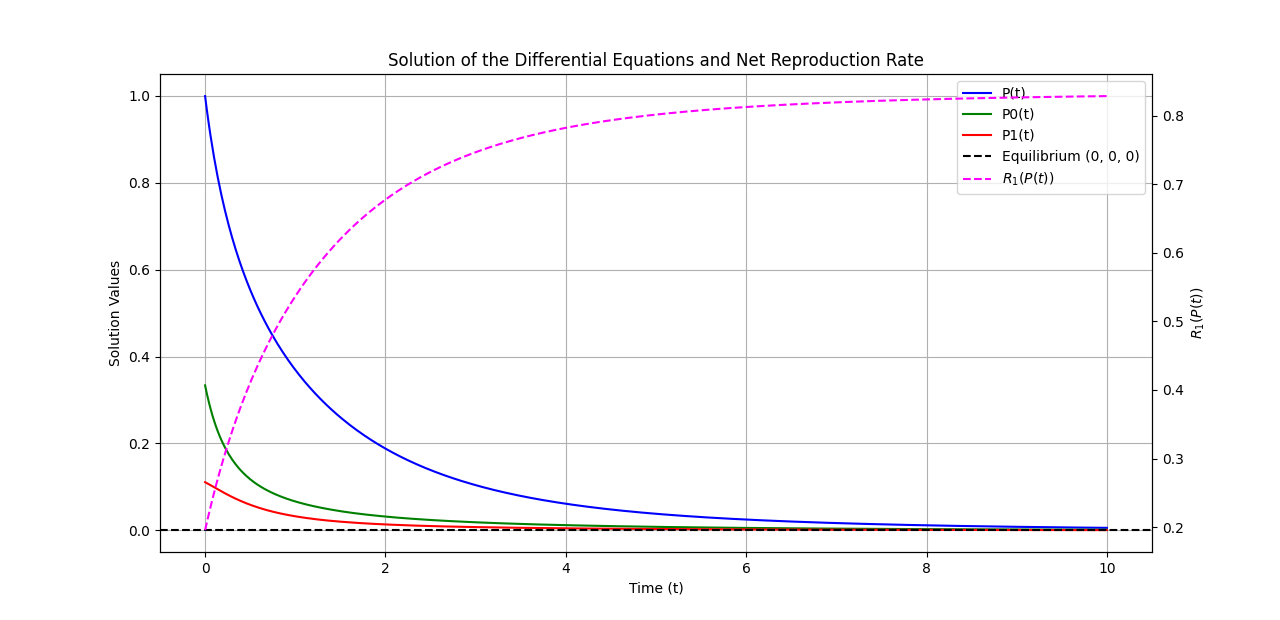}}
\caption{Population Trajectories and Net Reproduction Rate under System (%
\protect\ref{stud}) }
\label{fig:oneFigures}
\end{figure}

\begin{equation*}
\end{equation*}%
The comprehensive Python code implementing this methodology, developed with
the assistance of Microsoft Copilot in EDGE, is provided in the next section.

The results align with theoretical forecasts, underscoring the impact of age
structure on population development.

\section{Implementation \label{a}}

\subsection{Mathematical Explanation of the Process}

We begin with an age density function defined by

\begin{equation*}
p_{0}(a)=e^{-a},
\end{equation*}%
which describes a decaying profile with respect to the age variable $a$.
From this density, three integrals are computed to establish the initial
state:

\begin{equation*}
P_{0}=\int_{0}^{50}p_{0}(a)\,da,\quad
P_{00}=\int_{0}^{50}e^{-2a}\,p_{0}(a)\,da,\quad
P_{01}=\int_{0}^{50}a\,e^{-2a}\,p_{0}(a)\,da.
\end{equation*}%
These integrals yield the starting values for the state variables. Next, a
system of three coupled ordinary differential equations \ref{par}) is
considered for the evolution of the state over time. Let the functions $P(t)$%
, $P_{0}(t)$, and $P_{1}(t)$ denote three interdependent quantities whose
dynamics are governed by the system (\ref{par}). Here, the coefficients
depend nonlinearly on $P(t)$, which implies that the evolution of each
variable is influenced by both its own state and by the state of the others.
In addition, the net reproduction rate is defined as a function of $P(t)$ by
(\ref{nrp}). This expression quantifies the reproduction potential based on
the current value of $P(t)$.

After computing the evolution of the three state variables over a chosen
time interval, $P(t)$, $P_{0}(t)$, and $P_{1}(t)$ are plotted as functions
of time. Simultaneously, the reproduction rate $R_{1}(P(t))$ is calculated
alongside these variables and displayed on a secondary vertical axis. This
dual-axis visualization allows one to compare the dynamic behavior of the
system with the corresponding reproduction measure in a single combined
graphical representation.

\subsection{ Python implementation}

At the following address, %
\url{https://github.com/coveidragos/population-dynamics/blob/main/population.py}%
, we provide an annotated Python implementation that illustrates our
theoretical framework.

\section{Additional remark on the obtained results\label{fin}}

Theorems \ref{mp1}-\ref{mp3} and \ref{L1}-\ref{L3} remain valid if the
fertility function, $\beta (a,P(t))$, in (\ref{eq:fertility_function}) is
substituted with: 
\begin{equation}
\beta (a,P(t))=\sum_{i=0}^{n}\beta _{i}(P(t))e^{-\rho _{i}a}
\label{eq:fertility_function2}
\end{equation}%
where $\rho _{i}>0$ are parameters governing the exponential decay with age,
provided that all other conditions on the relevant parameters and functions
are satisfied. On this settings, the net reproduction rate, denoted by $%
R_{n}(x)$, is defined as: 
\begin{equation*}
R_{n}(x)=\sum_{i=0}^{n}\frac{\beta _{i}(x)}{\rho _{i}+\mu (x)}\text{ and }%
R_{0}=R_{n}(0)=\sum_{i=0}^{n}\frac{\beta _{i}(0)}{\rho _{i}+\mu (0)}.
\end{equation*}%
The computation of the solution to the renewal equation 
\begin{equation*}
p\left( 0,t\right) =\overset{n}{\underset{i=0}{\Sigma }}\beta _{i}\left(
P\left( t\right) \right) \int_{0}^{\infty }e^{-\rho _{i}\sigma }p\left(
\sigma ,t\right) d\sigma =\overset{n}{\underset{i=0}{\Sigma }}\beta
_{i}\left( P\left( t\right) \right) P_{i}\left( t\right) ,
\end{equation*}%
where 
\begin{equation*}
P_{i}\left( t\right) =\int_{0}^{\infty }e^{-\rho _{i}\sigma }p\left( \sigma
,t\right) d\sigma \text{, for }i=0,1,2,..n,
\end{equation*}%
is reduced to solving a system of $n+2$ differential equations (\ref{afin}),
where%
\begin{equation*}
\begin{tabular}{l}
$A\left( P\left( t\right) \right) =%
\begin{pmatrix}
-\mu (P(t)) & \beta _{0}(P(t)) & \cdots & \beta _{n}(P(t)) \\ 
0 & -\rho _{0}-\mu (P(t))+\beta _{0}(P(t)) & \cdots & \beta _{n}(P(t)) \\ 
\vdots & \vdots & \ddots & \vdots \\ 
0 & \beta _{0}(P(t)) & ... & -\rho _{n}-\mu (P(t))+\beta _{n}(P(t))%
\end{pmatrix}%
.$%
\end{tabular}%
\end{equation*}

\section{Future Directions \label{fd}}

The fractional reformulation of the nonlinear logistic model for
age-structured populations considered herein offers a fertile ground for
future exploration. In the resulting fractional model, the substitution of
the classical time derivative by a Caputo fractional derivative of order $%
\alpha \in (0,1)$ preserves the core analytical structure of the original
model. In particular, the existence and uniqueness result (cf. Theorem~\ref%
{mp3}) continues to hold under analogous assumptions on the fertility and
mortality functions; the proof follows by similar fixed-point arguments and
fractional Gr\H{o}nwall's inequalities (see, e.g., \cite%
{diethelm2010,kilbas2006,podlubny}).

Moreover, the equilibrium analysis based on the net reproduction condition

\begin{equation*}
R_{n}(P^{\ast })=\sum_{i=0}^{n}\,\frac{i!\beta _{i}(P^{\ast })}{\bigl(\rho
+\mu (P^{\ast })\bigr)^{i+1}}=1,
\end{equation*}%
remains valid in the fractional setting, thereby ensuring that the
characterization of nontrivial equilibria is retained. Likewise, the
asymptotic decay properties established in Theorem~\ref{L1} for the
classical system are expected to extend to the fractional model by employing
similar techniques adapted to nonlocal operators.

These pre-conceived results open several promising avenues for further
research. Future work may rigorously establish the fractional counterparts
of the main theorems using the analytical methods outlined above,
investigate bifurcation and stability phenomena as the fractional order
varies, and develop efficient numerical schemes for simulating the full
time-dependent dynamics. Detailed proofs and further arguments can be
recovered from the cited literature and from the comprehensive discussion
provided in the present article.

More future research directions, as highlighted in this article, may focus
on enhancing the model by incorporating stochastic processes, environmental
variability, and intergenerational dependencies. In addition, computational
advancements---including the implementation of numerical simulations and
machine learning techniques---could further refine the analysis of complex
population dynamics. Moreover, the article underscores that
interdisciplinary applications, such as those in epidemiology and
conservation biology, present exciting opportunities to expand the utility
of the proposed framework.

\section{Conclusions \label{CFD}}

This study rigorously examines a nonlinear logistic model tailored for
age-structured populations, incorporating interdependent fertility and
mortality functions. By addressing key challenges in the field, such as the
existence and stability of equilibrium solutions, and exploring the
asymptotic behavior of populations, this work provides a robust theoretical
framework for understanding long-term population dynamics.

The research resolves longstanding conjectures and extends the applicability
of age-structured models by considering nonlinear density-dependent effects.
The findings, which include the characterization of equilibrium states under
varying fertility and mortality rates, emphasize the importance of
demographic distributions in shaping population trajectories. Additionally,
the mathematical methods employed bridge the gap between theory and
practical applications, offering insights valuable for demographic analysis,
ecological modeling, and resource management.

By advancing the theoretical and practical understanding of nonlinear
age-structured population models, this study paves the way for future
innovations in mathematical biology and its applications in addressing
real-world challenges.

\section*{Declarations}

\begin{itemize}
\item Funding: Not applicable

\item Conflict of interest: No Conflict of interest.

\item Ethics approval and consent to participate: Not applicable.

\item Consent for publication: Not applicable.

\item Data availability: No datasets were generated or analysed during the
current study.

\item Materials availability: Not applicable

\item Code availability: Not applicable

\item Author contribution: Not applicable.
\end{itemize}

\end{document}